\documentclass[12pt,amsymb,fullpage]{amsart}
\usepackage{amssymb,amscd,pstricks}

\newtheorem{theorem}{Theorem}[section]
\newtheorem{defn}[theorem]{Definition}

\newtheorem{lemma}[theorem]{Lemma}

\newtheorem{eple}[theorem]{Example}
\newtheorem{rmk}[theorem]{Remarks}
\newtheorem{dsc}[theorem]{Discussion}
\newtheorem{nota}[theorem]{Notation}

\newsavebox{\indbin}
\savebox{\indbin}{\begin{picture}(0,0)
\newlength{\gnu}
\settowidth{\gnu}{$\smile$} \setlength{\unitlength}{.5\gnu}
\put(-1,-.65){$\smile$} \put(-.25,.1){$|$}
\end{picture}}

\newcommand{\be}{\begin{enumerate}}
\newcommand{\bd}{\begin{defn}}
\newcommand{\bt}{\begin{theorem}}
\newcommand{\bl}{\begin{lemma}}
\newcommand{\ee}{\end{enumerate}}
\newcommand{\ed}{\end{defn}}
\newcommand{\et}{\end{theorem}}
\newcommand{\el}{\end{lemma}}

\begin{document}
\title{Non Oscillatory Functions and A Fourier Inversion Theorem for Some Functions of Very Moderate Decrease}
\author{Tristram de Piro}
\address{Flat 3, Redesdale House, 85 The Park, Cheltenham, GL50 2RP }
 \email{t.depiro@curvalinea.net}
\thanks{}
\begin{abstract}
We consider non oscillatory functions and prove an everywhere Fourier Inversion Theorem for functions of very moderate decrease. The proofs rely on some ideas in nonstandard analysis.
\end{abstract}
\maketitle
Non oscillatory functions of very moderate decrease are quite typical in the class $L^{2}(\mathcal{R})$ and are important in Physics. The first result presented here improves slightly on a result due to Carlsen, see \cite{C}, in that the convergence is everywhere rather than almost everywhere. This property is useful in verifying certain differentiability criteria in Physics, arising mainly from Maxwell's equations, and in showing decay properties of fields produced in connection with Jefimenko's equations, see \cite{dep1} and \cite{dep3}. At the end of the paper, we show how the result can be improved to functions of just very moderate decrease, without the non oscillatory assumption. The proofs rely on some ideas from nonstandard analysis, contained in the papers \cite{dep2} and \cite{dep0}.

\begin{defn}
\label{decreasenonoscillate}
We say that $f\in C(\mathcal{R})$ is of very moderate decrease if there exists a constant $C\in\mathcal{R}_{>0}$ with $|f(x)|\leq {C\over |x|}$, for $|x|>1$. We say that $f\in C(\mathcal{R})$ is of moderate decrease if there exists a constant $C\in\mathcal{R}_{>0}$ with $|f(x)|\leq {C\over |x|^{2}}$, for $|x|>1$. We say that $f\in C(\mathcal{R})$ is non-oscillatory if there exist finitely many points $\{y_{i}:1\leq i\leq n\}$, for which $f|_{(y_{i},y_{i+1})}$ is monotone, $1\leq i\leq n-1$, and $f|_{(-\infty,y_{1})}$, $f|_{(y_{n},\infty)}$ are monotone. We say that $f\in C(\mathcal{R})$ is oscillatory if there exists an infinite sequence of points $\{y_{i}:i\in\mathcal{Z}\}$, for which $f|_{(y_{i},y_{i+1})}$, $i\in\mathcal{Z}$, is monotone and there exists $\delta\in\mathcal{R}_{>0}$, with $|y_{i+1}-y_{i}|\geq \delta$, for $i\in\mathcal{Z}$.  \\
\end{defn}

\begin{lemma}
\label{riemannlebesgue5}
Let $f\in C(\mathcal{R})$ and ${df\over dx}\in C(\mathcal{R})$ be of very moderate decrease, with $f$ and ${df\over dx}$ non-oscillatory, then defining the Fourier transform by;\\

$\mathcal{F}(f)(k)={1\over (2\pi)^{1\over 2}}lim_{r\rightarrow \infty}\int_{-r}^{r}f(y)e^{-iky}dy$ $(k\neq 0)$\\

$\mathcal{F}({df\over dx})(k)={1\over (2\pi)^{1\over 2}}lim_{r\rightarrow \infty}\int_{-r}^{r}{df\over dx}(y)e^{-iky}dy$ $(k\neq 0)$\\

we have that $\mathcal{F}(f)(k)$ and $\mathcal{F}({df\over dx})(k)$ are bounded, for $|k|>k_{0}>0$, and there exists a constant $G\in\mathcal{R}_{>0}$, such that;\\

$|\mathcal{F}(f)(k)|\leq {G\over |k|^{2}}$\\

for sufficiently large $k$.

\end{lemma}

\begin{proof}
As $f$ is of very moderate decrease, we have that $f$ is continuous and $lim_{|x|\rightarrow\infty}f(x)=0$. Similarly, ${df\over dx}$ is continuous and $lim_{|x|\rightarrow\infty}{df\over dx}=0$. As $lim_{|x|\rightarrow\infty}f(x)=0$, and $f$ is non-oscillatory, we have that, for $k\neq 0$, the indefinite integral;\\

$lim_{r\rightarrow\infty}\int_{-r}^{r}f(y)e^{-iky}dy$\\

$=lim_{r\rightarrow\infty}\int_{-r}^{r}f(y)cos(ky)dy-ilim_{r\rightarrow\infty}\int_{-r}^{r}f(y)sin(ky)dy$\\

exists. As $f$ is of very moderate decrease and non-oscillatory, we have that $|f(x)|\leq {D\over |x|}$, for $|x|>E$, and monotone in the intervals $(-\infty,E)$ and $(E,\infty)$. Using the method of \cite{dep0}, letting $K=max(|f||_{[-E,E]})$, we have that;\\

$|lim_{r\rightarrow\infty}\int_{-r}^{r}f(y)cos(ky)dy|\leq 2KE+2(|\int_{E}^{{\pi\over 2|k|}+{n_{k}\pi\over |k|}}{Dcos(ky)\over y}dy|+|\int_{{\pi\over 2|k|}+{n_{k}\pi\over |k|}}^{{\pi\over 2|k|}+{(n_{k}+1)\pi\over |k|}}{Dcos(ky)\over y}dy|)$\\

$\leq 2KE+{2D\over E}({\pi\over 2|k|}+{n_{k}\pi\over |k|}-E)+2\int_{E}^{E+{\pi\over |k|}}{Dsin(|k|(y-E))\over y}dy$\\

$\leq 2KE+{2D\pi\over E|k|}+2\int_{E}^{E+{\pi\over |k|}}{Dsin(|k|(y-E))\over y}dy$\\

$|lim_{r\rightarrow\infty}\int_{-r}^{r}f(y)sin(ky)dy|\leq 2KE+2(|\int_{E}^{{m_{k}\pi\over |k|}}{Dsin(ky)\over y}dy|+|\int_{{m_{k}\pi\over |k|}}^{{(m_{k}+1)\pi\over |k|}}{Dsin(ky)\over y}dy|)$\\

$\leq 2KE+{2D\over E}({m_{k}\pi\over |k|}-E)+2\int_{E}^{E+{\pi\over |k|}}{Dsin(|k|(y-E))\over y}dy$\\

$\leq 2KE+{2D\pi\over E|k|}+2\int_{E}^{E+{\pi\over |k|}}{Dsin(|k|(y-E))\over y}dy$\\

where $n_{k}=\mu n({\pi\over 2|k|}+{n\pi\over |k|}\geq E:n\in\mathcal{Z}_{\geq 0})$ and $m_{k}=\mu n({n\pi\over |k|}\geq E:n\in\mathcal{Z}_{\geq 0})$\\

so that;\\

$|lim_{r\rightarrow\infty}\int_{-r}^{r}f(y)e^{-iky}dy|\leq 4KE+{4D\pi\over E|k|}+4\int_{E}^{E+{\pi\over |k|}}{Dsin(|k|(y-E))\over y}dy$\\

$=4KE+{4D\pi\over E|k|}+4D([{-cos(|k|(y-E))\over |k|y}]^{E+{\pi\over |k|}}_{E}-\int_{E}^{E+{\pi\over |k|}}{cos(|k|(y-E))\over y^{2}}dy)$\\

$=4KE+{4D\pi\over E|k|}+4D({1\over E|k|+\pi}+{1\over E|k|}+\int_{E}^{E+{\pi\over |k|}}{cos(|k|(y-E))\over y^{2}}dy)$\\

$\leq 4KE+{4D\pi\over E|k|}+4D({1\over E|k|+\pi}+{1\over E|k|}+\int_{E}^{\infty}{1\over y^{2}}dy)$\\

$\leq 4KE+{4D\pi\over E|k|}+4D({2\over E|k|}+{1\over E})=N_{k}$\\

so that $\mathcal{F}(f)(k)$ and, similarly, $\mathcal{F}({df\over dx})(k)$ are bounded, for $|k|>k_{0}>0$.\\

We have, using integration by parts, that;\\

$\mathcal{F}({df\over dx})(k)={1\over (2\pi)^{1\over 2}}lim_{r\rightarrow\infty}\int_{-r}^{r}{df\over dx}(y)e^{-iky}dy$\\

$={1\over (2\pi)^{1\over 2}}lim_{r\rightarrow\infty}([f(y)e^{-iky}]^{r}_{-r}+ik\int_{-r}^{r}f(y)e^{-iky}dy)$\\

$={1\over (2\pi)^{1\over 2}}[f(y)e^{-iky}]^{\infty}_{-\infty}+ik{1\over (2\pi)^{1\over 2}}lim_{r\rightarrow\infty}\int_{-r}^{r}f(y)e^{-iky}dy$\\

$=ik\mathcal{F}(f)(k)$\\

so that, for $|k|>1$;\\

$|\mathcal{F}(f)(k)|\leq {|\mathcal{F}({df\over dx})(k)|\over |k|}$, $(\dag)$\\

As ${df\over dx}$ is continuous and non-oscillatory, by the proof of Lemma 0.9 in \cite{dep0}, using underflow, for $r\in\mathcal{R}_{>0}$, we can find $\{F_{r},G_{r}\}\subset\mathcal{R}_{>0}$, such that, for all $|k|>F_{r}$, we have that;\\

$|{1\over (2\pi)^{1\over 2}}\int_{-r}^{r}{df\over dx}(y) e^{-iky}dy|<{G_{r}\over |k|}$, $(**)$\\

It is easy to see from the proof, that $\{F_{r},G_{r}\}$ can be chosen uniformly in $r$. Then, from $(**)$, we obtain that, for $|k|>F$;\\

$|\mathcal{F}({df\over dx})(k)|<{G\over |k|}$, for $|k|>F$\\

and, from $(\dag)$, for $|k|>max(F,1)$, that;\\

$|\mathcal{F}(f)(k)|\leq {|\mathcal{F}({df\over dx})(k)|\over |k|}<{G\over |k|^{2}}$\\

\end{proof}

\begin{defn}
\label{approx}
Let $f\in C^{3}(\mathcal{R})$, with $f,f',f''$ and $f'''$ bounded, then we define an approximating sequence $\{f_{m}:m\in\mathcal{N}\}$ by the requirements;\\

$(i)$. $f_{m}\in C^{2}(\mathcal{R})$, for $m\in\mathcal{N}$.\\

$(ii)$. $f_{m}|_{[-m,m]}=f|_{[-m,m]}$.\\

$(iii)$. $f_{m}$ is of uniform moderate decay, in the sense that there exists a constant $C\in\mathcal{R}_{>0}$, independent of $m$, with;\\

$|f_{m}(x)|\leq {C\over |x|^{2}}$, for $x\in (-\infty,-m-{1\over m})\cup (m+{1\over m},\infty)$\\

$(iv)$. There exists constants $\{D,E\}\subset\mathcal{R}_{>0}$, with $\int_{-m-{1\over m}}^{m}|f_{m}(x)|dx\leq {D\over m}$  and $\int_{m}^{m+{1\over m}}|f_{m}(x)|dx\leq {D\over m}$.\\

\end{defn}
\begin{lemma}
\label{sequences}
Let $f\in C(\mathcal{R})$ and ${df\over dx}\in C(\mathcal{R})$ be of very moderate decrease, with $f$ and ${df\over dx}$ non-oscillatory. Let $\{f_{m};m\in\mathcal{N}\}$ be an approximating sequence. Let $\mathcal{F}$ be the ordinary Fourier transform, defined for each $f_{m}$, then, for any $k_{0}>0$, the sequence $\{\mathcal{F}(f_{m}):m\in\mathcal{N}\}$ converges pointwise and uniformly to $\mathcal{F}(f)$ on ${\mathcal{R}\setminus (|k|<k_{0})}$, where $\mathcal{F}(f)$ is defined in Lemma \ref{riemannlebesgue5}. In particularly, $\mathcal{F}(f)\in C({\mathcal{R}\setminus\{0\}})$.

\end{lemma}
\begin{proof}
For $g\in C(\mathcal{R})$ and $n\in\mathcal{N}$, define;\\

$\mathcal{F}_{n}(g)(k)={1\over (2\pi)^{1\over 2}}\int_{-n}^{n}f(y)e^{-iky}dy$\\

For $k\in{\mathcal{R}\setminus (|k|<k_{0})}$, $\{m,n\}\subset\mathcal{N}$, and $m\geq n$, $\epsilon>0$,$\delta>0$, we have;\\

$|\mathcal{F}(f)(k)-\mathcal{F}(f_{m})(k)|\leq |\mathcal{F}(f)(k)-\mathcal{F}_{n}(f)(k)|+|\mathcal{F}_{m}(f)(k)-\mathcal{F}_{m}(f_{m})(k)|$\\

$+|\mathcal{F}_{m}(f_{m})(k)-\mathcal{F}(f_{m})(k)|$\\

$=|\mathcal{F}(f)(k)-\mathcal{F}_{m}(f)(k)|+|\mathcal{F}_{m}(f_{m})(k)-\mathcal{F}(f_{m})(k)|$\\

$\leq |\mathcal{F}(f)(k)-\mathcal{F}_{m}(f)(k)|+\int_{-\infty}^{-m}|f_{m}(x)|dx+|+\int_{m}^{\infty}|f_{m}(x)|dx$\\

$=|\mathcal{F}(f)(k)-\mathcal{F}_{m}(f)(k)|+\int_{-\infty}^{-m-{1\over m}}|f_{m}(x)dx+\int_{-m-{1\over m}}^{-m}|f_{m}(x)|dx$\\

$+\int_{m}^{m+{1\over m}}|f_{m}(x)dx+\int_{m+{1\over m}}^{\infty}|f_{m}(x)|dx$\\

$\leq |\mathcal{F}(f)(k)-\mathcal{F}_{m}(f)(k)|+{D+E\over m}+\int_{-\infty}^{-m-{1\over m}}{C\over x^{2}}dx+\int_{m+{1\over m}}^{\infty}{C\over x^{2}}dx$\\

$\leq  |\mathcal{F}(f)(k)-\mathcal{F}_{m}(f)(k)|+{D+E\over m}+{2C\over m+{1\over m}}$\\

$\leq  |\mathcal{F}(f)(k)-\mathcal{F}_{m}(f)(k)|+{2C+D+E\over m}$\\

$\leq {C_{k_{0}}\over m}+{2C+D+E\over m}$\\

$\leq \epsilon+\delta$, for $m\geq max({C_{k_{0}}\over\epsilon},{2C+D+E\over \delta})$. As $\epsilon>0$ and $\delta>0$ were arbitrary, we obtain the first result. The fact that each $\mathcal{F}(f_{m})$ in the approximating sequence is continuous follows from $(iii)$ in definition \ref{approx} and the DCT.  The last result then follows immediately from the fact that $k_{0}>0$ is arbitrary and the uniform limit of continuous functions is continuous.

\end{proof}

\begin{lemma}
\label{polynomial2}
If $m\in\mathcal{R}_{>0}$ is sufficiently large, $\{a_{0},a_{1},a_{2}\}\subset\mathcal{R}$, there exists $h\in \mathcal{R}[x]$ of degree $5$, with the property that;\\

$h(m)=a_{0}$, $h'(m)=a_{1}$, $h''(m)=a_{2}$, $(i)$\\

$h(m+{1\over m})=h'(m+{1\over m})=h''(m+{1\over m})=0$ $(ii)$\\

$|h_{[m,m+{1\over m}]}|\leq C$\\

for some $C\in\mathcal{R}_{>0}$, independent of $m$ sufficiently large, and, if $h'''(m)>0$, $h'''(x)|_{[m,m+{1\over m}]}>0$, if $h'''(m)<0$, $h'''|_{[m,m+{1\over m}]}<0$. In particularly;\\

$\int_{m}^{m+{1\over m}}|h'''(x)|dx=|a_{2}|$\\

\end{lemma}
\begin{proof}
If $p(x)$ is any polynomial, we have that $h(x)=(x-(m+{1\over m}))^{3}p(x)$ satisfies condition $(ii)$. Then;\\

$h'(x)=3(x-(m+{1\over m}))^{2}p(x)+(x-(m+{1\over m}))^{3}p'(x)$\\

$h''(x)=6(x-(m+{1\over m}))p(x)+6(x-(m+{1\over m}))^{2}p'(x)+(x-(m+{1\over m}))^{3}p''(x)$\\

$h'''(x)=6p(x)+18(x-(m+{1\over m}))p'(x)+9(x-(m+{1\over m}))^{2}p''(x)$\\

so we can satisfy $(i)$, by requiring that;\\

$(a)$. $-{p(m)\over m^{3}}=a_{0}$\\

$(b)$. ${3p(m)\over m^{2}}-{p'(m)\over m^{3}}=a_{1}$\\

$(c)$. ${-6p(m)\over m}+{6p'(m)\over m^{2}}-{p''(m)\over m^{3}}=a_{2}$\\

which has the solution;\\

$p(m)=-a_{0}m^{3}$, $p'(m)=-3a_{0}m^{4}-a_{1}m^{3}$, $p''(m)=-12a_{0}m^{5}-6a_{1}m^{4}-a_{2}m^{3}$\\

and can be satisfied by the polynomial;\\

$p(x)={1\over 2}(-12a_{0}m^{5}-6a_{1}m^{4}-a_{2}m^{3})(x-m)^{2}$\\

$+(-3a_{0}m^{4}-a_{1}m^{3})(x-m)+(-a_{0}m^{3})$\\

$={1\over 2}(-12a_{0}m^{5}-6a_{1}m^{4}-a_{2}m^{3})x^{2}+(-m(-12a_{0}m^{5}-6a_{1}m^{4}-a_{2}m^{3})$\\

$+(-3a_{0}m^{4}-a_{1}m^{3}))x+({m^{2}\over 2}(-12a{0}m^{5}-6a_{1}m^{4}-a_{2}m^{3})$\\

$-m(-3a_{0}m^{4}-a_{1}m^{3})-a_{0}m^{3})$\\

$=(-6a_{0}m^{5}-3a_{1}m^{4}-{a_{2}\over 2}m^{3})x^{2}+(12a_{0}m^{6}+6a_{1}m^{5}+a_{2}m^{4}-3a_{0}m^{4}$\\

$-a_{1}m^{3})x+(-6a_{0}m^{7}-3a_{1}m^{6}-{a_{2}\over 2}m^{5}+3a_{0}m^{5}+a_{1}m^{4}-a_{0}m^{3})$\\

$=(-6a_{0}m^{5}-3a_{1}m^{4}-{a_{2}\over 2}m^{3})x^{2}+(12a_{0}m^{6}+6a_{1}m^{5}+(a_{2}-3a_{0})m^{4}$\\

$-a_{1}m^{3})x+(-6a_{0}m^{7}-3a_{1}m^{6}+(3a_{0}-{a_{2}\over 2})m^{5}+3a_{0}m^{5}+a_{1}m^{4}-a_{0}m^{3})$\\

$=ax^{2}+bx+c$ $(*)$\\

so that;\\

$h'''(x)=6(ax^{2}+bx+c)+18(x-(m+{1\over m}))(2ax+b)+9(x-(m+{1\over m}))^{2}2a$\\

$=(60a)x^{2}+(24b-72a(m+{1\over m}))x+(6c-18(m+{1\over m})b+18a(m+{1\over m})^{2})$\\

and, using the computation $(*)$\\

$h'''(x)=(60(-6a_{0}m^{5})+O(m^{4}))x^{2}+(24.12a_{0}m^{6}-72m(-6a_{0}m^{5})$\\

$+O(m^{5}))x+(6.-6a_{0}m^{7}-18m(12a_{0}m^{6})+18m^{2}(-6a_{0}m^{5})+O(m^{6}))$\\

$=(-360a_{0}m^{5}+O(m^{4}))x^{2}+(740a_{0}m^{6}+O(m^{5}))x+$\\

$(-360a_{0}m^{7}+O(m^{6}))$\\

which, by the quadratic formula, has roots when;\\

$x={-740a_{0}m^{6}+\-\sqrt{740^{2}a_{0}^{2}m^{12}-4(-360a_{0}m^{5})(-360a_{0}m^{7})}\over 2.-360a_{0}m^{5}}+O(1)$\\

$={740m\over 720}+\-{170m\over 720}+O(1)$\\

$={19m\over 24}+O(1)$ or ${91m\over 72}+O(1)$\\

We have that $m>{19m\over 24}$ and $m+{1\over m}<{91m\over 72}$ iff $m>\sqrt{72\over 19}$, and, clearly, we can ignore the $O(1)$ term for $m$ sufficiently large. In particularly, for sufficiently large $m$, $h'''(x)$ has no roots in the interval $[m,m+{1\over m}]$, so $h'''|_{[m,m+{1\over m}]}>0$ or $h'''|_{[m,m+{1\over m}]}<0$.\\

We calculate that;\\

$|h_{[m,m+{1\over m}]}|=|(x-(m+{1\over m}))^{3}p(x)|_{[m,m+{1\over m}]}|$\\

$\leq {1\over m^{3}}|p(x)|_{[m,m+{1\over m}]}$\\

$={1\over m^{3}}|[{1\over 2}(-12a_{0}m^{5}-6a_{1}m^{4}-a_{2}m^{3})(x-m)^{2}$\\

$+(-3a_{0}m^{4}-a_{1}m^{3})(x-m)+(-a_{0}m^{3})]|_{[m,m+{1\over m}]}$\\

$\leq {1\over m^{3}}[{1\over 2}|-12a_{0}m^{5}-6a_{1}m^{4}-a_{2}m^{3}|{1\over m^{2}}+|-3a_{0}m^{4}-a_{1}m^{3}|{1\over m}+|-a_{0}m^{3}|]$\\

$\leq {12|a_{0}|m^{5}+6|a_{1}|m^{4}+|a_{2}|m^{3}|\over m^{5}}+{3|a_{0}|m^{4}+|a_{1}|m^{3}\over m^{4}}+{|a_{0}|m^{3}\over m^{3}}$\\

$\leq 12|a_{0}|+6|a_{1}|+|a_{2}|+3|a_{0}|+|a_{1}|+|a_{0}|$ $(m>1)$\\

$\leq 16|a_{0}|+7|a_{1}|+|a_{2}|$\\

For the final claim, we have, as $h'''|_{[m,m+{1\over m}]}>0$ or $h'''|_{[m,m+{1\over m}]}<0$, that, using the fundamental theorem of calculus;\\

$\int_{m}^{m+{1\over m}}|h'''(x)|dx=|\int_{m}^{m+{1\over m}}h'''(x)dx|$\\

$=|h''(m+{1\over m})-h''(m)|=|-h''(m)|=|a_{2}|$\\
\end{proof}
\begin{lemma}
\label{differentialequation}
If $m\in\mathcal{R}_{>0}$, $\{a_{0},a_{1},a_{2},a_{3}\}\subset\mathcal{R}$, there exists $h\in C^{3}(\mathcal{R})$, with the property that;\\

$h(m)=a_{0}$, $h'(m)=a_{1}$, $h''(m)=a_{2}$,  $h'''(m)=a_{3}$, $(i)$\\

$h(m+{1\over m})=h'(m+{1\over m})=h''(m+{1\over m})=h'''(m+{1\over m})=0$ $(ii)$\\

$|h|_{[m,m+1]}\leq C$\\

where $C\in\mathcal{R}_{>0}$ is independent of $m>1$, and, if $a_{3}>0$, $h'''(x)|_{[m,m+{1\over m}]}\geq 0$, $a_{3}<0$, $h'''(x)|_{[m,m+{1\over m}]}\leq 0$. In particularly;\\

$\int_{m}^{m+{1\over m}}|h'''(x)|dx=|a_{2}|$\\

\end{lemma}
\begin{proof}
Let $g(x)$ be a polynomial, then it is clear that the polynomial $h_{1}(x)=(x-(m+{1\over m})^{n}g(x)$, for $n\geq 4$, has the property $(ii)$, that $h_{1}(m+{1\over m})=h_{1}'(m+{1\over m})=h_{1}''(m+{1\over m})=h_{1}'''(m+{1\over m})=0$. The condition $(i)$, then amounts to the equations;\\

$(i)'$ ${g(m)\over (-1)^{n}m^{n}}=a_{0}$\\

$(ii)'$ ${ng(m)\over (-1)^{n-1}m^{n-1}}+{g'(m)\over (-1)^{n}m^{n}}=a_{1}$\\

$(iii)'$ ${n(n-1)g(m)\over (-1)^{n-2}m^{n-2}}+{2ng'(m)\over (-1)^{n-1}m^{n-1}}+{g''(m)\over (-1)^{n}m^{n}}=a_{2}$\\

$(iv)'$ ${n(n-1)(n-2)g(m)\over (-1)^{n-3}m^{n-3}}+{3n(n-1)g'(m)\over (-1)^{n-2}m^{n-2}}+{3ng''(m)\over (-1)^{n-1}m^{n-1}}+{g'''(m)\over (-1)^{n}m^{n}}=a_{3}$\\

which we can solve, by requiring that;\\

$(i)''$ $g(m)=(-1)^{n}a_{0}m^{n}$\\

$(ii)''$ $g'(m)=(-1)^{n}a_{1}m^{n}+(-1)^{n}a_{0}nm^{n+1}$\\

$(iii)''$ $g''(m)=(-1)^{n}a_{2}m^{n}+2(-1)^{n}na_{1}m^{n+1}+(-1)^{n}n(n+1)a_{0}m^{n+2}$\\

$(iv)''$ $g'''(m)=(-1)^{n}a_{3}m^{n}+3n(-1)^{n}a_{2}m^{n+1}+(-1)^{n}a_{1}n(n+3)m^{n+2}$\\

$+n(n+1)(n+2)(-1)^{n}a_{0}m^{n+3}$ $(*)$\\

Let;\\

$g_{1}(x)=((-1)^{n}a_{3}m^{n}+3n(-1)^{n}a_{2}m^{n+1}+(-1)^{n}a_{1}n(n+3)m^{n+2}$\\

$+n(n+1)(n+2)(-1)^{n}a_{0}m^{n+3})(x-m)^{3}+((-1)^{n}a_{2}m^{n}+2(-1)^{n}na_{1}m^{n+1}$\\

$+(-1)^{n}n(n+1)a_{0}m^{n+2})(x-m)^{2}+((-1)^{n}a_{1}m^{n}+(-1)^{n}a_{0}nm^{n+1})$\\

$(x-m)+((-1)^{n}a_{0}m^{n})$\\

Then $g_{1}(x)$ satisfies $(*)$, and so does any function of the form $g_{2}(x)+g_{1}(x)$ where;\\

$g_{2}(m)=g_{2}'(m)=g_{2}''(m)=g_{2}'''(m)=0$\\

provided $g_{2}\in C^{3}(\mathcal{R})$. In this case, if;\\

$h(x)=(x-(m+{1\over m})^{n}(g_{2}(x)+g_{1}(x))$\\

then $h$ satisfies $(i),(ii)$. We have that;\\

$|x-(m+{1\over m})^{n}g_{1}(x)|_{[m,m+{1\over m}]}\leq {1\over m^{n}}(|g_{2}|_{[m,m+{1\over m}]}+|g_{1}|_{[m,m+{1\over m}]})$\\

$\leq {1\over m^{n}}(|g_{2}|_{[m,m+{1\over m}]}+{1\over m^{n}}|((-1)^{n}a_{3}m^{n}+3n(-1)^{n}a_{2}m^{n+1}+(-1)^{n}a_{1}n(n+3)m^{n+2}$\\

$+n(n+1)(n+2)(-1)^{n}a_{0}m^{n+3}){1\over m^{3}}+((-1)^{n}a_{2}m^{n}+2(-1)^{n}na_{1}m^{n+1}$\\

$+(-1)^{n}n(n+1)a_{0}m^{n+2}){1\over m^{2}}+((-1)^{n}a_{1}m^{n}+(-1)^{n}a_{0}nm^{n+1})$\\

${1\over m}+((-1)^{n}a_{0}m^{n})|$\\

$=|((-1)^{n}a_{3}m^{n}+3n(-1)^{n}a_{2}m^{n+1}+(-1)^{n}a_{1}n(n+3)m^{n+2}$\\

$+n(n+1)(n+2)(-1)^{n}a_{0}m^{n+3}){1\over m^{n+3}}+((-1)^{n}a_{2}m^{n}+2(-1)^{n}na_{1}m^{n+1}$\\

$+(-1)^{n}n(n+1)a_{0}m^{n+2}){1\over m^{n+2}}+((-1)^{n}a_{1}m^{n}+(-1)^{n}a_{0}nm^{n+1})$\\

${1\over m^{n+1}}+((-1)^{n}a_{0})|$\\

$\leq |a_{3}|+3n|a_{2}|+n(n+3)|a_{1}|+n(n+1)(n+2)|a_{0}|+|a_{2}|+2n|a_{1}|+n(n+1)|a_{0}|+|a_{1}|+n|a_{0}|+|a_{0}|$, $(m\geq 1)$\\

$={1\over m^{n}}(|g_{2}|_{[m,m+{1\over m}]}+(n+1)(n^{2}+3n+1)|a_{0}|+(n^{2}+5n+1)|a_{1}+(3n+1)|a_{2}|+|a_{3}|=F$ $(F)$\\

where $F\in\mathcal{R}_{>0}$ is independent of $m$. Using the product rule, the condition that $h'''(x)=0$ in the interval $(m,m+{1\over m})$, is given by;\\

$n(n-1)(n-2)(x-(m+{1\over m}))^{n-3}(g_{2}+g_{1})(x)+3n(n-1)(x-(m+{1\over m}))^{n-2}(g_{2}+g_{1})'(x)$\\

$+3n(x-(m+{1\over m}))^{n-1}(g_{2}+g_{1})''(x)+(x-(m+{1\over m}))^{n}(g_{2}+g_{1})'''(x)=0$\\

which, dividing by $(x-(m+{1\over m}))^{n-3}$, reduces to;\\

$n(n-1)(n-2)(g_{2}+g_{1})(x)+3n(n-1)(x-(m+{1\over m}))(g_{2}+g_{1})'(x)+$\\

$3n(x-(m+{1\over m}))^{2}(g_{2}+g_{1})''(x)+(x-(m+{1\over m}))^{3}(g_{2}+g_{1})'''(x)=0$\\

and;\\

$n(n-1)(n-2)g_{2}(x)+3n(n-1)(x-(m+{1\over m}))g_{2}'(x)+3n(x-(m+{1\over m}))^{2}g_{2}''(x)$\\

$+(x-(m+{1\over m}))^{3}g_{2}'''(x)=-(n(n-1)(n-2)g_{1}(x)+3n(n-1)(x-(m+{1\over m}))g_{1}'(x)$\\

$+3n(x-(m+{1\over m}))^{2}g_{1}''(x)+(x-(m+{1\over m}))^{3}g_{1}'''(x))$ $(A)$\\

Without loss of generality, assuming that;\\

$-(n(n-1)(n-2)g_{1}(x)+3n(n-1)(x-(m+{1\over m}))g_{1}'(x)+3n(x-(m+{1\over m}))^{2}g_{1}''(x)$\\

$+(x-(m+{1\over m}))^{3}g_{1}'''(x))|_{m}=-(n(n-1)(n-2)a_{0}-{3n(n-1)a_{1}\over m}+{3na_{2}\over m^{2}}$\\

$-{a_{3}\over m^{3}}\geq 0$\\

we can choose an analytic function $\phi(x)$ on $[m,m+{1\over m}]$ with;\\

$(a)$. $\phi(x)\leq -(n(n-1)(n-2)g_{1}(x)+3n(n-1)(x-(m+{1\over m}))g_{1}'(x)+3n(x-(m+{1\over m}))^{2}g_{1}''(x)$\\

$+(x-(m+{1\over m}))^{3}g_{1}'''(x))$\\

$(b)$. $\phi(m)=0$\\

The third order differential equation for $g_{2}$;\\

$n(n-1)(n-2)g_{2}(x)+3n(n-1)(x-(m+{1\over m}))g_{2}'(x)+3n(x-(m+{1\over m}))^{2}g_{2}''(x)$\\

$+(x-(m+{1\over m}))^{3}g_{2}'''(x)=\phi(x)$, on $[m,1+m]$ $(B)$\\

with the requirement that $g_{2}(m)=g_{2}'(m)=g_{2}''(m)=0$, has a solution in $C^{3}([m,m+{1\over m}))$ by Peano's existence theorem. By the fact $(b)$, we must have that $g_{2}'''(m)=0$. Writing the power series for $\phi$ on $[m,m+{1\over m}]$, as;\\

$\phi(x)=\sum_{j=0}^{\infty}b_{j}(x-(m+{1\over m}))^{j}$\\

we can use the method of equating coefficients, to obtain a particular solution, with;\\

$g_{2,part}(x)=\sum_{j=0}^{\infty}a_{j,part}(x-(m+{1\over m}))^{j}$, with;\\

$a_{j,part}={b_{j}\over n(n-1)(n-2)+3n(n-1)j+3nj(j-1)+j(j-1)(j-2)}$, $(j\geq 3)$\\

$a_{2,part}={b_{2}\over n(n-1)(n-2)+6n(n-1)+3n}$ $a_{1,part}={b_{1}\over n(n-1)(n-2)+3n(n-1)}$ $a_{0,part}={b_{0}\over n(n-1)(n-2)}$\\

so that $g_{2,part}$ is analytic as $|a_{j,0}|\leq {|b_{j}|\over n(n-1)(n-2)}$ for $j\geq 0$.\\

To solve the homogenous Euler equation;\\

$n(n-1)(n-2)g_{2}(x)+3n(n-1)(x-(m+{1\over m}))g_{2}'(x)+3n(x-(m+{1\over m}))^{2}g_{2}''(x)$\\

$+(x-(m+{1\over m}))^{3}g_{2}'''(x)=0$ on $[m,m+{1\over m}]$\\

we can make the substitution $y=m+{1\over m}-x$, to reduce to the equation;\\

$n(n-1)(n-2)g_{2,m}(y)+3n(n-1)yg_{2,m}'(y)+3ny^{2}g_{2,m}''(y)+y^{3}g_{2,m}'''(y)=0$ on $[0,{1\over m}]$\\

with $g_{2,m}(y)=g_{2}(m+{1\over m}-y)$. Making the further substitution $y=e^{u}$, and letting $r_{2,m}(u)=g_{2,m}(e^{u})$, we have that;\\

$r_{2,m}'(u)=g_{2,m}'(e^{u})e^{u}$\\

$r_{2,m}''(u)=g_{2,m}''(e^{u})e^{2u}+g_{2,m}'(e^{u})e^{u}$\\

$r_{2,m}'''(u)=g_{2,m}'''(e^{3u})+3g_{2,m}''(e^{u})e^{2u}+g_{2,m}'(e^{u})e^{u}$\\

so that;\\

$n(n-1)(n-2)g_{2,m}(e^{u})+3n(n-1)e^{u}g_{2,m}'(e^{u})+3ne^{2u}g_{2,m}''(e^{u})+e^{3u}g_{2,m}'''(e^{u})$\\

$=n(n-1)(n-2)r_{2,m}(u)+3n(n-1)e^{u}(r_{2,m}'(u)e^{-u})+3ne^{2u}((r_{2,m}''(u)-g_{2,m}'(e^{u})e^{u})e^{-2u})$\\

$+e^{3u}((r_{2,m}'''(u)-3g_{2,m}''(e^{u})e^{2u}-g_{2,m}'(e^{u})e^{u})e^{-3u})$\\

$=n(n-1)(n-2)r_{2,m}(u)+3n(n-1)r_{2,m}'(u)+3nr_{2,m}''(u)-3ng_{2,m}'(e^{u})e^{u}+r_{2,m}'''(u)-3g_{2,m}''(e^{u})e^{2u}$\\

$-g_{2,m}'(e^{u})e^{u}$\\

$=n(n-1)(n-2)r_{2,m}(u)+3n(n-1)r_{2,m}'(u)+3nr_{2,m}''(u)+r_{2,m}'''(u)-(3n+1)g_{2,m}'(e^{u})e^{u}-3g_{2,m}''(e^{u})e^{2u}$\\

$=n(n-1)(n-2)r_{2,m}(u)+3n(n-1)r_{2,m}'(u)+3nr_{2,m}''(u)+r_{2,m}'''(u)-(3n+1)r_{2,m}'(u)$\\

$-3e^{2u}((r_{2,m}''(u)-g_{2,m}'(e^{u})e^{u})e^{-2u})$\\

$=n(n-1)(n-2)r_{2,m}(u)+(3n^{2}-6n-1)r_{2,m}'(u)+3nr_{2,m}''(u)+r_{2,m}'''(u)-3r_{2,m}''(u)+3g_{2,m}'(e^{u})e^{u}$\\

$=n(n-1)(n-2)r_{2,m}(u)+(3n^{2}-6n-1)r_{2,m}'(u)+3(n-1)r_{2,m}''(u)+r_{2,m}'''(u)+3r_{2,m}'(u)$\\

$=n(n-1)(n-2)r_{2,m}(u)+(3n^{2}-6n+2)r_{2,m}'(u)+(3n-3)r_{2,m}''(u)+r_{2,m}'''(u)=0$ $(C)$\\

We have that;\\

$(\lambda^{3}+3(n-1)\lambda^{2}+(3n^{2}-6n+2)\lambda+n(n-1)(n-2))'=3\lambda^{2}+6(n-1)\lambda+(3n^{2}-6n+2)$\\

which has roots when $\lambda=-(n-1)+\-{1\over \sqrt{3}}$, so that, for large n, the characteristic polynomial of $(C)$ has exactly one real root $\lambda_{1}$ and 2 complex conjugate non-real roots, $\{\lambda_{2}+i\lambda_{3},\lambda_{2}-i\lambda_{3}\}$. It follows, the general solution of $(C)$ is given by;\\

$r_{2,m}(u)=A_{1}e^{\lambda_{1}u}+A_{2}e^{\lambda_{2}u+i\lambda_{3}}+A_{3}e^{\lambda_{2}u-i\lambda_{3}}$\\

where $\{A_{1},A_{2},A_{3}\}\subset\mathcal{C}$, and, we can obtain a real solution, fitting the corresponding initial conditions, of the form;\\

$r_{2,m}(u)=B_{1}e^{\lambda_{1}u}+B_{2}e^{\lambda_{2}u}cos(\lambda_{3}u)+B_{3}e^{\lambda_{2}u}sin(\lambda_{3}u)$\\

where $\{B_{1},B_{2},B_{3}\}\subset\mathcal{R}$. It follows that;\\

$g_{2,m}(y)=r_{2,m}(ln(y))$\\

$g_{2}(x)=g_{2,m}(m+{1\over m}-x)+g_{2,part}(x)=r_{2,m}(ln(m+{1\over m}-x))+g_{2,part}(x)$\\

$=B_{1}e^{\lambda_{1}ln(m+{1\over m}-x)}+B_{2}e^{\lambda_{2}ln(m+{1\over m}-x)}cos(\lambda_{3}ln(m+{1\over m}-x))$\\

$+B_{3}e^{\lambda_{2}ln(m+{1\over m}-x)}sin(\lambda_{3}ln(m+{1\over m}-x))+g_{2,part}(x)$ (on $[m,m+{1\over m}]$)\\

We have that;\\

$\lambda_{1}|\lambda_{2}+i\lambda_{3}|^{2}=-n(n-1)(n-2)$\\

$\lambda_{1}+\lambda_{2}+i\lambda_{3}+\lambda_{2}-i\lambda_{3}=\lambda_{1}+2\lambda_{2}=-3(n-1)$\\

Computing the highest degree in $n$ term of the characteristic polynomial, we obtain that, for $\lambda=\alpha n$;\\

$\alpha^{3}n^{3}+3n(\alpha n)^{2}+3n^{2}(\alpha n)+n^{3}=n^{3}(\alpha+3)^{3}=0$\\

so that $\alpha=-3$, $\lambda_{1}=-3n+O(1)$ and $2\lambda_{2}=-3(n-1)-(-3n+O(1))=3-O(1)=O(1)$\\

Then, if $B_{1}=0$, we can see that $g_{2}(x)$ has at most a ${1\over x^{O(1)}}$ singularity at $(m+{1\over m})$, which we can achieve with a 2-parameter family choice for the initial conditions of $\{\phi(m),\phi'(m),\phi''(m)\}$. If;\\

$-(n(n-1)(n-2)a_{0}-{3n(n-1)a_{1}\over m}+{3na_{2}\over m^{2}}-{a_{3}\over m^{3}}\neq 0$\\

we can clearly achieve this, while satisfying $(a),(b)$. If;\\

$-(n(n-1)(n-2)a_{0}-{3n(n-1)a_{1}\over m}+{3na_{2}\over m^{2}}-{a_{3}\over m^{3}}=0$\\

by requiring the the additional property $(c)$;\\

$\phi'(m)<-(n(n-1)(n-2)g_{1}(x)+3n(n-1)(x-(m+{1\over m}))g_{1}'(x)+3n(x-(m+{1\over m}))^{2}g_{1}''(x)$\\

$+(x-(m+{1\over m}))^{3}g_{1}'''(x))'|_{m}$\\

we can clearly satisfy $(a),(b)$ as well.\\

Then, as, for sufficiently large $n$;\\

$lim_{x\rightarrow 0}({B_{2}x^{n}\over x^{O(1)}}sin(\lambda_{3}ln(x))+{B_{3}x^{n}\over x^{O(1)}}cos(\lambda_{3}ln(x)))$\\

$=lim_{x\rightarrow 0}({B_{2}x^{n}\over x^{O(1)}}sin(\lambda_{3}ln(x))+{B_{3}x^{n}\over x^{O(1)}}cos(\lambda_{3}ln(x)))'$\\

$=lim_{x\rightarrow 0}({B_{2}x^{n}\over x^{O(1)}}sin(\lambda_{3}ln(x))+{B_{3}x^{n}\over x^{O(1)}}cos(\lambda_{3}ln(x)))''$\\

$=lim_{x\rightarrow 0}({B_{2}x^{n}\over x^{O(1)}}sin(\lambda_{3}ln(x))+{B_{3}x^{n}\over x^{O(1)}}cos(\lambda_{3}ln(x)))'''=0$\\

we obtain that $(x-(m+{1\over m}))^{n}g_{2}(x)$ extends to a solution in $C^{3}([m,m+{1\over m}])$, and $(x-(m+{1\over m}))^{n}(g_{2}+g_{1})(x)\in C^{3}([m,m+{1\over m}])$. By the fact $(a)$, $(A)$ has no solutions in $(m,m+{1\over m})$, so that $h'''(x)\geq 0$.\\

We have that;\\

$|(x-(m+{1\over m})^{n}g_{2}(x)|_{[m,m+{1\over m}]}=|(x-(m+{1\over m})^{n}(B_{2}e^{\lambda_{2}ln(m+{1\over m}-x)}cos(\lambda_{3}ln(m+{1\over m}-x))$\\

$+B_{3}e^{\lambda_{2}ln(m+{1\over m}-x)}sin(\lambda_{3}ln(m+{1\over m}-x))+g_{2,part}(x))|$\\

$\leq |B_{2}|m^{\lambda_{2}-n}+|B_{3}|m^{\lambda_{2}-n}+m^{-n}|g_{2,part}(x)|$\\

Noting the right hand side of $(a)$ is bounded by $O(m^{n})$ on $[m,m+{1\over m}]$, we can also choose $\phi(x)$ and $g_{2,part}(x)$ to be of $O(m^{n})$ on $[m,m+{1\over m}]$, irrespective of the choice of initial conditions $\{\phi(m),\phi'(m),\phi''(m)\}$. We have that $\phi'(m)=O(m^{n+1})$, in the special case, so that choosing $\{B_{2},B_{3}\}$ sufficiently small, noting;\\

$(x-(m+{1\over m})^{n}(B_{2}e^{\lambda_{2}ln(m+{1\over m}-x)}cos(\lambda_{3}ln(m+{1\over m}-x))$\\

$+B_{3}e^{\lambda_{2}ln(m+{1\over m}-x)}sin(\lambda_{3}ln(m+{1\over m}-x)))'|_{m}=O(max(B_{2}m^{n-\lambda_{2}-1},B_{3}m^{n-\lambda_{2}-1}))$\\

we can assume that;\\

$|(x-(m+{1\over m})^{n}g_{2}(x)|_{[m,m+{1\over m}]}\leq D$\\

where $D\in\mathcal{R}_{>0}$ is independent of $m$, so that, using $(F)$;\\

$|h(x)|_{[m,m+{1\over m}]}\leq |(x-(m+{1\over m})^{n}g_{1}(x)|_{[m,m+{1\over m}]}+|(x-(m+{1\over m})^{n}g_{2}(x)|_{[m,m+{1\over m}]}\leq F+D$\\

For the final claim, we have, as $h'''|_{[m,m+{1\over m}]}\geq 0$ or $h'''|_{[m,m+{1\over m}]}\leq 0$, that, using the fundamental theorem of calculus, that;\\

$\int_{m}^{m+{1\over m}}|h'''(x)|dx=|\int_{m}^{m+{1\over m}}h'''(x)dx|$\\

$=|h''(m+{1\over m})-h''(m)|=|-h''(m)|=|a_{2}|$\\

\end{proof}
\begin{lemma}
\label{approx1}
Let $f$ be as in Definition \ref{approx}, then there exists an approximating sequence $\{f_{m}:m\in\mathcal{N}\}$. Moreover, for sufficiently large $m$, $|\mathcal{F}(f_{m})(k)|\leq {Cm\over |k|^{3}}$, for $|k|>1$, where $C\in\mathcal{R}_{>0}$, independent of $m$.\\

\end{lemma}
\begin{proof}
Define $f_{m}$ by setting;\\

$f_{m}(x)=f(x)$ for $x\in [-m,m]$\\

$f_{m}(x)=h_{1,m}(x)$, for $x\in [-m-{1\over m},-m]$\\

$f_{m}(x)=h_{2,m}(x)$, for $x\in [m,m+{1\over m}]$\\

$f_{m}(x)=0$, for $x\in (-\infty,-m-{1\over m}]$\\

$f_{m}(x)=0$, for $x\in [m+{1\over m},\infty)$\\

where $\{h_{1,m},h_{2,m}\}\subset C^{2}([-m-{1\over m},-m]\cup [m,m+{1\over m}])$ are generated by the data $a_{1,m,0}=f(-m)$, $a_{1,m,1}=f'(-m)$, $a_{1,m,2}=f''(-m)$, $a_{2,m,0}=f(m)$, $a_{2,m,1}=f'(m)$, $a_{2,m,2}=f''(m)$, guaranteed by Lemma \ref{polynomial2} or Lemma \ref{differentialequation}. By the constructions of Lemmas \ref{polynomial2} and \ref{differentialequation}, we have that $(i)$ in Definition \ref{approx} holds. By the definition, we have $(ii)$. As $f_{m}$ is identically zero on $(-\infty,-m-{1\over m}]\cup [m+{1\over m},\infty)$, we have that $(iii)$ holds. By the proof of Lemma \ref{polynomial2}, or using Lemma \ref{differentialequation}, we have that;\\

 $max(|h_{1,m}|_{[m,m+{1\over m}]}|,|h_{2,m}|_{[-m-{1\over m},-m]}|)\leq 16||f||_{\infty}+7||f'||_{\infty}+||f''||_{\infty}$\\

 It follows that;\\

 $\int_{-m-{1\over m}}^{-m}|f_{m}(x)|dx \leq (16||f||_{\infty}+7||f'||_{\infty}+||f''||_{\infty})(-m-(-m-{1\over m}))$\\

 $\leq {D\over m}$\\

 $\int_{m}^{m+{1\over m}}|f_{m}(x)|dx \leq (16||f||_{\infty}+7||f'||_{\infty}+||f''||_{\infty})((m+{1\over m})-m)$\\

 $\leq {E\over m}$\\

 where $D=E=(16||f||_{\infty}+7||f'||_{\infty}+||f''||_{\infty})$\\

 proving $(iv)$. For the second claim, we have that;\\

 $\mathcal{F}(f_{m}''')(k)={1\over (2\pi)^{1\over 2}}\int_{-\infty}^{\infty}f_{m}'''(x)e^{-ikx}dx$\\

 $={1\over (2\pi)^{1\over 2}}([f_{m}''(x)e^{-ikx}]^{\infty}_{-\infty}-ik\int_{-\infty}^{\infty}f_{m}''(x)e^{-ikx}dx$\\

 $={-ik\over (2\pi)^{1\over 2}}([f_{m}'(x)e^{-ikx}]^{\infty}_{-\infty}-ik\int_{-\infty}^{\infty}f_{m}'(x)e^{-ikx}dx)$\\

 $={-k^{2}\over (2\pi)^{1\over 2}}([f_{m}(x)e^{-ikx}]^{\infty}_{-\infty}-ik\int_{-\infty}^{\infty}f_{m}(x)e^{-ikx}dx)$\\

 $={ik^{3}\over (2\pi)^{1\over 2}}\int_{-\infty}^{\infty}f_{m}(x)e^{-ikx}dx)$\\

 so that, for $|k|>1$;\\

 $|\mathcal{F}(f_{m})(k)|=|{1\over (2\pi)^{1\over 2}}\int_{-\infty}^{\infty}f_{m}(x)e^{-ikx}dx|$\\

 $={|{1\over (2\pi)^{1\over 2}}\int_{-\infty}^{\infty}f_{m}'''(x)e^{-ikx}dx|\over |k|^{3}}$\\

$\leq {1\over |k|^{3}(2\pi)^{1\over 2}}\int_{-\infty}^{\infty}|f_{m}'''(x)e^{-ikx}|dx$\\

$={1\over |k|^{3}(2\pi)^{1\over 2}}\int_{-\infty}^{\infty}|f_{m}'''(x)|dx$\\

$={1\over |k|^{3}(2\pi)^{1\over 2}}(\int_{-m-{1\over m}}^{-m}|h_{1,m}'''(x)|dx+\int_{-m}^{m}|f'''(x)|dx+\int_{m}^{m+{1\over m}}|h_{2,m}'''(x)|dx)$\\

$\leq {1\over |k|^{3}(2\pi)^{1\over 2}}(|f''(-m)|+2m|f'''|_{\infty}+|f''(m))$\\

$\leq {1\over |k|^{3}(2\pi)^{1\over 2}}(2||f''||_{\infty}+2m||f'''||_{\infty})$\\

$\leq {1\over |k|^{3}(2\pi)^{1\over 2}}(2m+2m||f'''||_{\infty})$, $(m>||f''||_{\infty}$\\

$={Cm\over |k|^{3}}$\\

where $C={1\over (2\pi)^{1\over 2}}(2+2||f'''||_{\infty})$\\

\end{proof}

\begin{lemma}
\label{inversion}
Let $f\in C^{3}(\mathcal{R})$, with $f$ and ${df\over dx}$ non-oscillatory and of very moderate decrease, with $\{f,f',f'',f'''\}$ bounded, then $\mathcal{F}(f)\in L^{1}(\mathcal{R})$, and we have that;\\

$f(x)=\mathcal{F}^{-1}(\mathcal{F}(f))(x)$\\

where, for $g\in L^{1}(\mathcal{R})$;\\

$\mathcal{F}^{-1}(g)(x)={1\over (2\pi)^{1\over 2}}\int_{-\infty}^{\infty}g(k)e^{ikx}dk$\\
\end{lemma}

\begin{proof}
By Lemma \ref{riemannlebesgue5}, we have that there exists $C\in\mathcal{R}_{>0}$, with $|\mathcal{F}(f)(k)|\leq {C\over |k|^{2}}$, for sufficiently large $k$, $(*)$. As $f$ is of very moderate decrease, we have that $|f|^{2}\leq {D\over |x|^{2}}$, for $|x|>1$, so that, as $f\in C^{0}(\mathcal{R})$, we have that $f\in L^{2}(\mathcal{R})$. It follows that $\mathcal{F}(f)\in L^{2}(\mathcal{R})$, and $\mathcal{F}(f)|_{[-n,n]}\in L^{1}(\mathcal{R})$, for any $n\in\mathcal{N}$, $(**)$. Combining $(*),(**)$, we obtain that $\mathcal{F}(f)\in L^{1}(\mathcal{R})$. Let $\{f_{m}:m\in\mathcal{N}\}$ be the approximating sequence, given by Lemma \ref{approx1}, then, as $f_{m}\in L^{1}(\mathcal{R})$, $\mathcal{F}(f_{m})$ is continuous and, by Lemma \ref{sequences}, converges uniformly to $\mathcal{F}(f)$ on ${\mathcal{R}\setminus \{0\}}$. It follows that $\mathcal{F}(f)\in C^{0}({\mathcal{R}\setminus \{0\}})$. As $f_{m}\in C^{2}(\mathcal{R})$ and $f_{m}''\in L^{1}(\mathcal{R})$, we have that there exists constants $D_{m}\in\mathcal{R}_{>0}$, such that $|\mathcal{F}(f_{m})(k)\leq {D_{m}\over |k|^{2}}$, for sufficiently large $k$. Moreover, as $x^{n}f_{m}(x)\in L^{1}(\mathcal{R})$, for $n\in\mathcal{N}$, $\mathcal{F}(f_{m})\in C^{\infty}(\mathcal{R})$. It follows, the Fourier inversion theorem $f_{m}=\mathcal{F}^{-1}(\mathcal{F}(f_{m}))$, $(***)$, holds for each $f_{m}$, see the proof in \cite{dep2}. By Lemma \ref{sequences}, we have that, for $k_{0}>0$, $|\mathcal{F}(f)(k)-\mathcal{F}(f_{m})(k)|\leq {E_{k_{0}}\over m}$, for $|k|>k_{0}$. As $f$ is of very moderate decrease, we have that, $f-f_{m}\in L^{2}(\mathcal{R})$ and $||f-f_{m}||_{L^{2}(\mathcal{R})}\rightarrow 0$, $||\mathcal{F}(f)-\mathcal{F}(f_{m})||_{L^{2}(\mathcal{R})}\rightarrow 0$  as $m\rightarrow \infty$, in particularly for $m$ sufficiently large, we have that $||\mathcal{F}(f)-\mathcal{F}(f_{m})||_{L^{2}(\mathcal{R})}\leq 1$. We then have that, for $\epsilon>0$;\\

$||\mathcal{F}(f)-\mathcal{F}(f_{m})||_{L^{1}((-\epsilon,\epsilon)}$\\

$=<|\mathcal{F}(f)-\mathcal{F}(f_{m})|,1_{(-\epsilon,\epsilon)}>_{L^{2}((-\epsilon,\epsilon))}$\\

$\leq ||\mathcal{F}(f)-\mathcal{F}(f_{m})||_{L^{2}((-\epsilon,\epsilon))}||1_{(-\epsilon,\epsilon)}||_{L^{2}((-\epsilon,\epsilon))}$\\

$=\sqrt{2}\epsilon^{1\over 2}||\mathcal{F}(f)-\mathcal{F}(f_{m})||_{L^{2}((-\epsilon,\epsilon))}$\\

$\leq \sqrt{2}\epsilon^{1\over 2}$ $(A)$\\

for $m$ sufficiently large. As $f$ is oscillatory and of very moderate decrease, $f$ is monotone in the intervals $(-\infty,-E)$ and $(E,\infty)$, and, we have that for $|k|>\epsilon$, $m>max(E,1)$, $f$ is monotone in the intervals $(-\infty,-m)$ and $(m,\infty)$, with $|f|_{(-\infty,-m)}|\leq {C\over m}$, $|f|_{(m,\infty)}|\leq {C\over m}$, for some $C\in\mathcal{R}_{>0}$. Then, using the proof of the alternating series test, see \cite{S} ;\\

$|\mathcal{F}(f)(k)-\mathcal{F}(m)(k)|={1\over (2\pi)^{1\over 2}}|lim_{r\rightarrow\infty}\int_{m<|x|<r}f(x)e^{-ikx}dx|$\\

$\leq {1\over (2\pi)^{1\over 2}}|lim_{r\rightarrow\infty}\int_{m<|x|<r}f(x)cos(kx)dx|+{1\over (2\pi)^{1\over 2}}|lim_{r\rightarrow\infty}\int_{m<|x|<r}f(x)sin(kx)dx|$\\

$={1\over (2\pi)^{1\over 2}}|\int_{m}^{{\pi(n_{k}+{1\over 2})\over |k|}}f(x)cos(kx)dx+\int_{-{\pi(n_{k}+{1\over 2})\over |k|}}^{-m}f(x)cos(kx)dx+\int_{m}^{{m_{k}\pi\over |k|}}f(x)sin(kx)dx$\\

$+\int_{-{m_{k}\pi\over |k|}}^{-m}f(x)sin(kx)dx+lim_{r\rightarrow\infty}\int_{{\pi(n_{k}+{1\over 2})\over |k|}<|x|<r}f(x)cos(kx)dx$\\

$+lim_{r\rightarrow\infty}\int_{{m_{k}\pi\over |k|}<|x|<r}f(x)sin(kx)dx|$\\

$\leq {4\over (2\pi)^{1\over 2}}{C\over m}{\pi\over |k|}+{1\over (2\pi)^{1\over 2}}|lim_{r\rightarrow\infty}\int_{{\pi(n_{k}+{1\over 2})\over |k|}<|x|<r}f(x)cos(|k|x)dx|$\\

$+{1\over (2\pi)^{1\over 2}}|lim_{r\rightarrow\infty}\int_{{m_{k}\pi\over |k|}<|x|<r}f(x)sin(|k|x)dx|$\\

$\leq {4\over (2\pi)^{1\over 2}}{C\over m}{\pi\over k_{0}}+{1\over (2\pi)^{1\over 2}}|\int_{{\pi(n_{k}+{1\over 2})\over |k|}}^{{\pi(n_{k}+{3\over 2})\over |k|}}f(x)cos(|k|x)dx|+{1\over (2\pi)^{1\over 2}}|\int_{{m_{k}\pi\over |k|}}^{{(m_{k}+1)\pi\over |k|}}f(x)sin(|k|x)dx|$\\

$\leq {4\over (2\pi)^{1\over 2}}{C\over m}{\pi\over k_{0}}+{4\over (2\pi)^{1\over 2}}{C\over m}{\pi\over k_{0}}$\\

$={8\pi^{1\over 2}C\over 2^{1\over 2}k_{0}m}={E_{k_{0}}\over m}$ $(B)$\\

where $n_{k}=\mu n(n\in\mathcal{Z}_{>0},{\pi(n+{1\over 2})\over |k|}\geq m)$, $m_{k}=\mu n(n\in\mathcal{Z}_{>0},{\pi n\over |k|}\geq m)$. Then, for $n\in\mathcal{N}$, $m\in\mathcal{N}$, with $m=n^{3\over 2}$, using Lemma \ref{approx1} and $(A),(B)$, we have, for $x\in\mathcal{R}$, that;\\

$|\mathcal{F}^{-1}(\mathcal{F}(f))(x)-\mathcal{F}^{-1}(\mathcal{F}(f_{m}))(x)|=|\mathcal{F}^{-1}(\mathcal{F}(f)(k)-\mathcal{F}(f_{m})(k))|$\\

$={1\over (2\pi)^{1\over 2}}|\int_{-n}^{n}(\mathcal{F}(f)(k)-\mathcal{F}(f_{m})(k))e^{ikx}dk+\int_{|k|>n}(\mathcal{F}(f)(k)-\mathcal{F}(f_{m})(k))e^{ikx}dk|$\\

$\leq {1\over (2\pi)^{1\over 2}}(\int_{-n}^{n}|\mathcal{F}(f)(k)-\mathcal{F}(f_{m})(k)|dk+\int_{|k|>n}|\mathcal{F}(f)(k)|dk+\int_{|k|>n}|\mathcal{F}(f_{m})(k)|dk)$\\

$\leq {1\over (2\pi)^{1\over 2}}(\int_{-\epsilon}^{\epsilon}|\mathcal{F}(f)(k)-\mathcal{F}(f_{m})(k)|dk+{2nE_{\epsilon}\over m}+\int_{|k|>n}{C\over |k|^{2}}dk+\int_{|k|>n}{Cm\over |k|^{3}}dk)$\\

$\leq {1\over (2\pi)^{1\over 2}}(\sqrt{2}\epsilon^{1\over 2}+{2nE_{\epsilon}\over n^{3\over 2}}+{2C\over n}+{Cn^{3\over 2}\over n^{2}})$\\

$<2\epsilon^{1\over 2}$\\

for sufficiently large $n$, so that, as $\epsilon>0$ was arbitrary, for $x\in\mathcal{R}$;\\

$lim_{m\rightarrow\infty}\mathcal{F}^{-1}(\mathcal{F}(f_{m}))(x)=\mathcal{F}^{-1}\mathcal{F}(f)(x)$, $(****)$\\

and, by Definition \ref{approx}, $(***)$, $(****)$;\\

$f(x)=lim_{m\rightarrow\infty}f_{m}(x)=lim_{m\rightarrow\infty}\mathcal{F}^{-1}(\mathcal{F}(f_{m}))(x)=\mathcal{F}^{-1}\mathcal{F}(f)(x)$\\

\end{proof}

\begin{rmk}
\label{carlesonkorner}
The previous lemma proves an inversion theorem for non-oscillatory functions with very moderate decrease. Such functions belong to $L^{2}(\mathcal{R})$ and an analogous result for Fourier series can be found in \cite{C}, where convergence is proved almost everywhere rather than everywhere. The corresponding result for transforms is that;\\

If $f\in L^{p}(\mathcal{R})$, $p\in (1,2]$, then;\\

$f(x)=lim_{R\rightarrow\infty}{1\over 2\pi}\int_{|k|\leq R}\mathcal{F}(f)(k)e^{ixk}dk$\\

for almost every $x\in\mathcal{R}$.\\

We can define the function $\mathcal{F}_{1}(f)(k)$, for $k\in\mathcal{R}$, using the usual Fourier transform transform method, when $f\in L^{1}(\mathcal{R})$, see \cite{SS}, and, we can define the function $\mathcal{F}_{2}(f)(k)$, for $k\in\mathcal{R}_{\neq 0}$, using the limit definition when $f$ is of very moderate decrease and non-oscillatory, a particular case of $f\in L^{2}(\mathcal{R})$. However, the operators $\mathcal{F}_{1}$ and $\mathcal{F}_{2}$ need not commute, so that even if we show that $\mathcal{F}_{2}\circ \mathcal{F}_{1}=Id$, it doesn't necessarily follow that $\mathcal{F}_{1}\circ\mathcal{F}_{2}=Id$. The first claim is, in a sense, shown in \cite{K};\\

If $f\in L^{1}(\mathcal{R})\cap C^{0}(\mathcal{R})$ and $|\mathcal{F}(f)(k)|\leq {A\over |k|}$, for all $k\neq 0$, and $A\in\mathcal{R}_{\geq 0}$, then;\\

$f(x)=lim_{R\rightarrow\infty}{1\over 2\pi}\int_{|k|\leq R}\mathcal{F}(f)(k)e^{ixk}dk$\\

for every $x\in\mathcal{R}$.

\end{rmk}

\begin{defn}
\label{analyticinfinity}
We say that $f:\mathcal{R}\rightarrow\mathcal{R}$ is analytic at infinity, if $f({1\over x})$ has a convergent power series expansion for $|x|<\epsilon$, $\epsilon>0$. We say that $f$ is eventually monotone, if there exists $y_{0}\in\mathcal{R}_{>0}$ such that $f|_{(-\infty,-y_{0})}$ and $f|_{(y_{0},\infty)}$ are monotone.\\

\end{defn}

\begin{rmk}
\label{physics}
The class of functions which are analytic at infinity and of very moderate decrease is important in Physics. The components of the causal field generated
by Jefimenko equations can be shown to have this property if the corresponding charge and current $(\rho,\overline{J})$ are smooth and have compact support.
\end{rmk}

A criteria for a function being non-oscillatory or eventually monotone is provided by the following lemma;\\

\begin{lemma}
\label{finitely}
If $f:\mathcal{R}\rightarrow\mathcal{R}$, $f\neq 0$ is analytic and analytic at infinity, then it has finitely many zeroes. If $f:\mathcal{R}\rightarrow\mathcal{R}$, ${df\over dx}$ is analytic and analytic at infinity, and $f\neq c$, where $c\in\mathcal{R}$, then $f$ is non-oscillatory. If $f:\mathcal{R}\rightarrow\mathcal{R}$,  $f$ is analytic for $|x|>a$, where $a\in\mathcal{R}_{\geq 0}$, analytic at infinity, and $f|_{|x|>a}\neq 0$ then $f$ has finitely many zeroes in the region $|x|>a+1$. If $f:\mathcal{R}\rightarrow\mathcal{R}$, ${df\over dx}$ is analytic for $|x|>a$, analytic at infinity, and $f|_{|x|>a}\neq c$, where $c\in\mathcal{R}$, then $f$ is eventually monotone.
\end{lemma}

\begin{proof}
For the first claim, suppose that $f$ has infinitely many zeroes. Then we can find a sequence $\{y_{i};i\in\mathcal{N}\}$ with $f(y_{i})=0$. If the sequence is bounded, then by the Bolzano-Weierstrass Theorem, we can find a subsequence $\{y_{i_{k}};k\in\mathcal{N}\}$, with $f(y_{i_{k}})=0$, converging to $y\in\mathcal{R}$. By continuity, we have that $f(y)=0$ and $y$ is a limit point of zeroes. As $f$ is analytic, by the identity theorem, it must be identically zero, contradicting the hypothesis. If the sequence is unbounded, then we can find a subsequence $\{y_{i_{k}};k\in\mathcal{N}\}$, with $f(y_{i_{k}})=0$, such that $lim_{k\rightarrow \infty}y_{i_{k}}=\infty$ or $lim_{k\rightarrow \infty}y_{i_{k}}=-\infty$. As $f$ is analytic at $\infty$, we can find $\epsilon>0$, such that $f(y)=0$ for $|y|>{1\over \epsilon}$. By the identity theorem again, $f$ is identically zero, contradicting the hypothesis. It follows that $f$ has finitely many zeroes. For the second claim, as ${df\over dx}\neq 0$, by the first part, there exist finitely many points $\{y_{1},\ldots,y_{n}\}$, with ${df\over dx}|_{y_{i}}=0$, for $1\leq i\leq n$, and with $y_{i}<y_{i+1}$, for $1\leq i\leq n-1$. In particularly, we have that $f|_{(-\infty,y_{1})}$, $f|_{(y_{n},\infty)}$ and $f|_{(y_{i},y_{i+1})}$ is monotone for $1\leq i\leq n-1$, so that $f$ is non-oscillatory. For the third claim, suppose that $f$ has infinitely many zeroes in the region $|x|>a+1$, then we can find a sequence $\{y_{i};i\in\mathcal{N}\}$ with $f(y_{i})=0$ and $|y_{i}>a+1$. As above, if the sequence is bounded, we can find a subsequence $\{y_{i_{k}};k\in\mathcal{N}\}$, with $f(y_{i_{k}})=0$, converging to $y\in\mathcal{R}$, with $|y|\geq a+1>a$. As $f$ is analytic for $|x|>a$, by the identity theorem, it must be identically zero in the region $|x|>a$, contradicting the hypothesis. If the sequence is unbounded, by the same argument as above, $f$ must be identically zero in the region $|x|>a$, contradicting the hypothesis. It follows that $f$ has finitely many zeroes in the region $|x|>a+1$. For the fourth claim, as ${df\over dx}|_{|x|>a}\neq 0$, by the first part, there exist finitely many points $\{y_{1},\ldots,y_{n}\}$, with ${df\over dx}|_{y_{i}}=0$, and $|y_{i}|>a+1$, for $1\leq i\leq n$. Choose $y_{0}>max_{1\leq i\leq n}(|y_{i}|)$, then ${df\over dx}|_{|x|>y_{0}}\neq 0$, so that $f|_{|x|>y_{0}}$ is monotone.

\end{proof}

\begin{lemma}
\label{limit}
If $f\in C^{0}(\mathcal{R})$ and $f$ is of very moderate decrease, then defining the Fourier transform by;\\

$\mathcal{F}(f)(k)={1\over (2\pi)^{1\over 2}}lim_{r\rightarrow \infty}\int_{-r}^{r}f(y)e^{-iky}dy$\\

we have that $\mathcal{F}(f)(k)$ defines a function $\mathcal{F}(f)(k)\in L^{2}(\mathcal{R})$, with $||\mathcal{F}(f)||_{L^{2}(\mathcal{R})}=||f||_{L^{2}(\mathcal{R})}$. There exists a sequence $\{r_{n}:n\in\mathcal{N}\}$ with $r_{n}\in\mathcal{R}_{>0}$ such that;\\

${1\over (2\pi)^{1\over 2}}lim_{n\rightarrow \infty}\int_{-r_{n}}^{r_{n}}f(y)e^{-iky}dy$\\

converges almost everywhere.\\

\end{lemma}

\begin{proof}
By the hypotheses, we have that $f\in L^{2}(\mathcal{R})$ and if $f_{r}=f\chi_{(-r,r)}$, where $\chi_{(=r,r)}$ is the characteristic function on $(-r,r)$, then $f_{r}\in L^{1}(\mathcal{R})\cap L^{2}(\mathcal{R})$ and $||f-f_{r}||_{L^{2}(\mathcal{R})}\rightarrow 0$, as $r\rightarrow \infty$, in particular the sequence $\{f_{r}:r\in\mathcal{R}\}$ is Cauchy. By a result in \cite{F}, we have that the usual Fourier transform $\mathcal{F}:L^{1}(\mathcal{R})\cap L^{2}(\mathcal{R})\rightarrow L^{2}(\mathcal{R})$ is an isometry, so that the sequence $\{\mathcal{F}(f_{r}):r\in\mathcal{R}\}$ is also Cauchy. By the completeness of $L^{2}(\mathcal{R})$, there exists $g\in L^{2}(\mathcal{R})$, such that $||g-\mathcal{F}(f_{r})||_{L^{2}(\mathcal{R})}\rightarrow 0$, as $r\rightarrow\infty$. We have that;\\

$||\mathcal{F}(f)(k)||_{L^{2}(\mathcal{R})}=||lim_{r\rightarrow \infty}\mathcal{F}(f_{r})||_{L^{2}(\mathcal{R})}$\\

$=lim_{r\rightarrow\infty}||\mathcal{F}(f_{r})||_{L^{2}(\mathcal{R})}$\\

$=lim_{r\rightarrow\infty}||f_{r}||_{L^{2}(\mathcal{R})}$\\

$=||lim_{r\rightarrow\infty}f_{r}||_{L^{2}(\mathcal{R})}$\\

$=||f||_{L^{2}(\mathcal{R})}$\\

Finally, by a result in \cite{SE}, If we choose $r_{n}\in\mathcal{R}_{>0}$ such that $||f-f_{r_{n}}||_{L^{2}(\mathcal{R})}\leq 2^{-n}$, then $||\mathcal{F}(f)-\mathcal{F}(f_{r_{n}})||_{L^{2}(\mathcal{R})}\leq 2^{-n}$ and $|\mathcal{F}(f_{r_{n}})(k)-\mathcal{F}(f)(k)|\rightarrow 0$ almost everywhere.\\

\end{proof}
\begin{rmk}
\label{ae}
We will use $\mathcal{F}(f)(k)$ to denote this a.e limit.
\end{rmk}

\begin{lemma}
\label{riemannlebesgue5new}
Let $f\in C^{2}(\mathcal{R})$ and $\{f,{df\over dx}\}\subset C(\mathcal{R})$ be of very moderate decrease, with ${d^{2}f\over dx^{2}}$ of moderate decrease, then there exists a constant $G\in\mathcal{R}_{>0}$, such that;\\

$|\mathcal{F}(f)(k)|\leq {G\over |k|^{2}}$\\

for $|k|>1$. In particularly, $\mathcal{F}(f)(k)\in L^{1}(\mathcal{R})$ is defined everywhere for $k\neq 0$.\\

Let $f\in C^{1}(\mathcal{R})$ be of very moderate decrease, with ${df\over dx}$ of very moderate decrease and oscillatory, then there exists a constant $G\in\mathcal{R}_{>0}$, such that;\\

$|\mathcal{F}(f)(k)|\leq {G\over |k|^{2}}$\\

for sufficiently large $|k|$. In particularly, $\mathcal{F}(f)(k)\in L^{1}(\mathcal{R})$ is defined for $k\neq 0$.\\

\end{lemma}

\begin{proof}

We have that ${d^{2}f\over dx^{2}}\in L^{1}(\mathcal{R})$ and;\\

$|\mathcal{F}({d^{2}f\over dx^{2}})(k)|=|{1\over (2\pi)^{1\over 2}}\int_{-\infty}^{\infty}{d^{2}f\over dx^{2}}(y)e^{-iky}dy|$\\

$\leq {1\over (2\pi)^{1\over 2}}\int_{-\infty}^{\infty}|{d^{2}f\over dx^{2}}(y)e^{-iky}|dy$\\

$\leq {1\over (2\pi)^{1\over 2}}\int_{-\infty}^{\infty}|{d^{2}f\over dx^{2}}(y)|dy$\\

$={1\over (2\pi)^{1\over 2}}||{d^{2}f\over dx^{2}}||_{L^{1}(\mathcal{R})}$\\

$\leq G$\\

where $G\in\mathcal{R}_{>0}$. Then, using integration by parts and the DCT, we have, for $k\neq 0$, that;\\

$\mathcal{F}({d^{2}f\over dx^{2}})(k)={1\over (2\pi)^{1\over 2}}lim_{r\rightarrow\infty}\int_{-r}^{r}{d^{2}f\over dx^{2}}(y)e^{-iky}dy$\\

$={1\over (2\pi)^{1\over 2}}lim_{r\rightarrow\infty}([{df\over dx}(y)e^{-iky}]^{r}_{-r}+ik\int_{-r}^{r}{df\over dx}(y)e^{-iky}dy)$\\

$={1\over (2\pi)^{1\over 2}}[{df\over dx}e^{-iky}]^{\infty}_{-\infty}+ik{1\over (2\pi)^{1\over 2}}lim_{r\rightarrow\infty}\int_{-r}^{r}{df\over dx}e^{-iky}dy$\\

$={ik\over (2\pi)^{1\over 2}}lim_{r\rightarrow\infty}\int_{-r}^{r}{df\over dx}e^{-iky}dy$\\

$={ik\over (2\pi)^{1\over 2}}(lim_{r\rightarrow\infty}([f(y)e^{-iky}]^{r}_{-r}+ik\int_{-r}^{r}f(y)e^{-iky}dy))$\\

$=-k^{2}({1\over (2\pi)^{1\over 2}}lim_{r\rightarrow\infty}\int_{-r}^{r}f(y)e^{-iky}dy$\\

$=-k^{2}\mathcal{F}(f)(k)$\\

The calculation shows that $\mathcal{F}(f)(k)$ and $\mathcal{F}({df\over dx})(k)$ are defined for $k\neq 0$, and, for $|k|>1$;\\

$|\mathcal{F}(f)(k)|\leq {|\mathcal{F}({d^{2}f\over dx^{2}})(k)|\over |k|^{2}}$\\

$\leq {G\over |k|^{2}}$\\

We have that $\mathcal{F}(f)(k)\in L^{2}(\mathcal{R})$, so that $\mathcal{F}(f)|_{(-1,1)}(k)\in L^{2}(\mathcal{R})\subset L^{1}(\mathcal{R})$, and $\mathcal{F}(f)(k)\in L^{1}(\mathcal{R})$\\.

For the second claim, we have that, ${df\over dx}\in L^{1}(\mathcal{R})$, and, similarly to the above, $|\mathcal{F}({df\over dx})(k)|\leq H$, for some $H\in\mathcal{R}_{>0}$ and $k\in\mathcal{R}$. Using the DCT, for $k\neq 0$;\\

$\mathcal{F}({df\over dx})(k)={1\over (2\pi)^{1\over 2}}lim_{r\rightarrow\infty}\int_{-r}^{r}{df\over dx}(y)e^{-iky}dy$\\

$={1\over (2\pi)^{1\over 2}}lim_{r\rightarrow\infty}([f(y)e^{-iky}]^{r}_{-r}+ik\int_{-r}^{r}f(y)e^{-iky}dy)$\\

$={1\over (2\pi)^{1\over 2}}[f(y)e^{-iky}]^{\infty}_{-\infty}+ik{1\over (2\pi)^{1\over 2}}lim_{r\rightarrow\infty}\int_{-r}^{r}f(y)e^{-iky}dy$\\

$={ik\over (2\pi)^{1\over 2}}lim_{r\rightarrow\infty}\int_{-r}^{r}f(y)e^{-iky}dy$\\

$=ik\mathcal{F}(f)(k)$\\

showing that $\mathcal{F}(f)(k)$ is defined for $k\neq 0$, and, for $|k|>1$;\\

$|\mathcal{F}(f)(k)|\leq {|\mathcal{F}({df\over dx})(k)|\over |k|}$ $(*)$\\

$\leq {H\over |k|}$\\

We have that ${df\over dx}|_{(-1,1)}$ is non-oscillatory, and, moreover, there are at most ${2\over \delta}$ monotone intervals. Using the proof in \cite{dep0}, we get, for sufficiently large $|k|$;\\

$|\int_{-1}^{1}{df\over dx}e^{-iky}dy|<{E_{1}\over |k|}$\\

 where $E_{1}\leq {4C\pi\over \delta}$ and $C=max_{x\in\mathcal{R}}|{df\over dx}|$. As ${df\over dx}$ is of moderate decrease, we have that $|{df\over dx}|\leq {D\over |x|^{2}}$, for $|x|>1$, where $D\in\mathcal{R}_{>0}$.  Then, using the proof in \cite{dep0} again, and the definition of oscillatory, we have that, for sufficiently large $|k|$;\\

$|\int_{|y|>1}{df\over dx}e^{-iky}dy|<({2\pi\over |k|}\sum_{|y_{i}|>1}{D\over |y_{i}|^{2}})$\\

$\leq ({4\pi\over |k|}\sum_{n\in\mathcal{Z}_{\geq 0}}{D\over (y_{i_{0}}+n\delta)^{2}})$\\

$\leq {4\pi D\over \delta|k|}\int_{y_{i_{0}}}^{\infty}{dx\over x^{2}}$\\

$={4\pi D\over \delta|k|y_{i_{0}}}$\\

$\leq {4\pi D\over \delta|k|}$\\

where $|y_{i_{0}}|\geq 1$ and $|y_{i_{0}}|\leq |y_{i}|$, for all $|y_{i}|\geq N$. It follows that;\\

$|\mathcal{F}({df\over dx})(k)|=|\int_{-1}^{1}{df\over dx}e^{-iky}dy+\int_{|y|>1}{df\over dx}e^{-iky}dy|$\\

$\leq |\int_{-1}^{1}{df\over dx}e^{-iky}dy|+|\int_{|y|>1}{df\over dx}e^{-iky}dy|$\\

$\leq {E_{1}\over |k|}+{4\pi D\over \delta|k|}$\\

$\leq {4C\pi\over \delta|k|}+{4\pi D\over \delta|k|}$\\

$={R\over |k|}$ $(**)$\\

where $R={4\pi(C+D)\over \delta}$. Combining $(*)$ and $(**)$, we obtain that, for sufficiently large $|k|$;\\

$|\mathcal{F}(f)(k)|\leq {|\mathcal{F}({df\over dx})(k)|\over |k|}$\\

$\leq {G\over |k|^{2}}$\\

where $G={1\over (2\pi)^{1\over 2}}R$\\

\end{proof}

\begin{lemma}
\label{improvement}
Let $f\in C^{3}(\mathcal{R})$, with $\{f,{df\over dx}\}$ of very moderate decrease, ${d^{2}f\over dx^{2}}$ of very moderate decrease, and $\{f,f',f'',f'''\}$ bounded, then $\mathcal{F}(f)\in L^{1}(\mathcal{R})$, and we have that;\\

$f(x)=\mathcal{F}^{-1}(\mathcal{F}(f))(x)$\\

where, for $g\in L^{1}(\mathcal{R})$;\\

$\mathcal{F}^{-1}(g)(x)={1\over (2\pi)^{1\over 2}}\int_{-\infty}^{\infty}g(k)e^{ikx}dk$\\

The same result holds if $f\in C^{3}(\mathcal{R})$, with $f$ of very moderate decrease, ${df\over dx}$ of very moderate decrease and oscillatory, and $\{f,f',f'',f'''\}$ bounded.\\

\end{lemma}
\begin{proof}
By Lemma \ref{riemannlebesgue5new}, we have that there exists $E\in\mathcal{R}_{>0}$, with $|\mathcal{F}(f)(k)|\leq {E\over |k|^{2}}$, for sufficiently large $k$, $(*)$. As $f$ is of very moderate decrease, we have that $|f|^{2}\leq {D\over |x|^{2}}$, for $|x|>1$, so that, as $f\in C^{0}(\mathcal{R})$, we have that $f\in L^{2}(\mathcal{R})$. It follows that $\mathcal{F}(f)\in L^{2}(\mathcal{R})$, and $\mathcal{F}(f)|_{[-n,n]}\in L^{1}(\mathcal{R})$, for any $n\in\mathcal{N}$, $(**)$. Combining $(*),(**)$, we obtain that $\mathcal{F}(f)\in L^{1}(\mathcal{R})$. Let $\{f_{m}:m\in\mathcal{N}\}$ be the approximating sequence, given by Lemma \ref{approx1}, then, as $f_{m}\in L^{1}(\mathcal{R})$, $\mathcal{F}(f_{m})$ is continuous. As $f_{m}\in C^{2}(\mathcal{R})$ and $f_{m}''\in L^{1}(\mathcal{R})$, we have that there exists constants $D_{m}\in\mathcal{R}_{>0}$, such that $|\mathcal{F}(f_{m})(k)\leq {D_{m}\over |k|^{2}}$, for sufficiently large $k$. Moreover, as $x^{n}f_{m}(x)\in L^{1}(\mathcal{R})$, for $n\in\mathcal{N}$, $\mathcal{F}(f_{m})\in C^{\infty}(\mathcal{R})$. It follows, the Fourier inversion theorem $f_{m}=\mathcal{F}^{-1}(\mathcal{F}(f_{m}))$, $(***)$, holds for each $f_{m}$, see the proof in \cite{dep2}. As $f$ is of very moderate decrease, we have that, $f-f_{m}\in L^{2}(\mathcal{R})$ and $||f-f_{m}||_{L^{2}(\mathcal{R})}\rightarrow 0$, $||\mathcal{F}(f)-\mathcal{F}(f_{m})||_{L^{2}(\mathcal{R})}\rightarrow 0$  as $m\rightarrow \infty$. We then have that, for $n\in\mathcal{N}$, $m\in\mathcal{N}$;\\

$|\int_{-n}^{n}(\mathcal{F}(f)(k)-\mathcal{F}(f_{m})(k))dk|$\\

$\leq (2n)^{1\over 2}||\mathcal{F}(f-f_{m})||_L^{2}$\\

$\leq (2n)^{1\over 2}||f-f_{m}||_{L^{2}}$\\

$=(2n)^{1\over 2}(\int_{|x|>m}f^{2}(x)dx)^{1\over 2}$\\

$\leq (2n)^{1\over 2}(\int_{|x|>m}({D^{2}\over x^{2}})(x)dx)^{1\over 2}$\\

$=(2n)^{1\over 2}({2D^{2}\over m})^{1\over 2}$\\

$=2D{n^{1\over 2}\over m^{1\over 2}}$ $(A)$\\

where $D\in\mathcal{R}_{>0}$. Then, for $n\in\mathcal{N}$, $m\in \mathcal{N}$, with $m=[n^{3\over 2}]$, using Lemma \ref{approx1} and $(A)$, we have, for $x\in\mathcal{R}$, that;\\

$|\mathcal{F}^{-1}(\mathcal{F}(f))(x)-\mathcal{F}^{-1}(\mathcal{F}(f_{m}))(x)|=|\mathcal{F}^{-1}(\mathcal{F}(f)(k)-\mathcal{F}(f_{m})(k))|$\\

$={1\over (2\pi)^{1\over 2}}|\int_{-n}^{n}(\mathcal{F}(f)(k)-\mathcal{F}(f_{m})(k))e^{ikx}dk+\int_{|k|>n}(\mathcal{F}(f)(k)-\mathcal{F}(f_{m})(k))e^{ikx}dk|$\\

$\leq {1\over (2\pi)^{1\over 2}}(|\int_{-n}^{n}|\mathcal{F}(f)(k)-\mathcal{F}(f_{m})(k)|dk|+\int_{|k|>n}|\mathcal{F}(f)(k)|dk+\int_{|k|>n}|\mathcal{F}(f_{m})(k)|dk)$\\

$\leq {1\over (2\pi)^{1\over 2}}(2D{n^{1\over 2}\over m^{1\over 2}}+\int_{|k|>n}{E\over |k|^{2}}dk+\int_{|k|>n}{Cm\over |k|^{3}}dk)$\\

$\leq {1\over (2\pi)^{1\over 2}}(2D{n^{1\over 2}\over m^{1\over 2}}+{2E\over n}+{Cn^{3\over 2}\over n^{2}})$\\

$\leq {1\over (2\pi)^{1\over 2}}(2D({n\over [n^{3\over 2}]})^{1\over 2}+{2E\over n}+{Cn^{3\over 2}\over n^{2}})$\\

$<\epsilon$\\

for sufficiently large $n$, so that, as $\epsilon>0$ was arbitrary, for $x\in\mathcal{R}$;\\

$lim_{m\rightarrow\infty}\mathcal{F}^{-1}(\mathcal{F}(f_{m}))(x)=\mathcal{F}^{-1}\mathcal{F}(f)(x)$, $(****)$\\

and, by Definition \ref{approx}, $(***)$, $(****)$;\\

$f(x)=lim_{m\rightarrow\infty}f_{m}(x)=lim_{m\rightarrow\infty}\mathcal{F}^{-1}(\mathcal{F}(f_{m}))(x)=\mathcal{F}^{-1}\mathcal{F}(f)(x)$\\

\end{proof}

\end{document}